\documentclass[12pt,a4paper]{amsart}
\usepackage{graphicx}
\usepackage{times}
\usepackage{latexsym}
\usepackage{amssymb}
\usepackage{hyperref}
\usepackage{xcolor}
\usepackage{tikz}
\usepackage[all]{xy}

\newtheorem{prop}{Proposition}[section]

\newtheorem{thm}[prop]{Theorem}
\newtheorem{lemma}[prop]{Lemma}
\newtheorem{dfn}[prop]{Definition}

\theoremstyle{definition}

\newcommand{\C}[1]{{\mathcal #1}}
\newcommand{\B}[1]{{\mathbb #1}}

\newcommand{\bs}[1]{\boldsymbol{#1}}

\newcommand{\id}{\operatorname{\rm id}}

\newcommand{\Up}{{\tt Up}}

\newcommand{\op}{{\text{op}}}

\newcommand{\ord}[1]{\underline{#1}}

\renewcommand{\leq}{\leqslant}
\renewcommand{\geq}{\geqslant}

\newcommand{\cpN}{\mathbb{P}^N\!(\mathbb{C})}
\newcommand{\qcpN}{{\B{P}^N(\mathcal{T})}}
\newcommand{\Tplz}{{\mathcal{T}}}
\newcommand{\Cpct}{{\mathcal{K}}}
\newcommand{\CS}{{C(S^1)}}
\newcommand{\Tpn}[1]{{\Tplz^{\otimes {#1}}}}

\newcommand{\Jirred}{{\mathcal{J}}}
\newcommand{\Mirred}{{\mathcal{M}}}

\newcommand{\mN}{{\mathbb{N}}}

\newcommand{\sw}[1]{{{}^{(#1)}}}
\newcommand{\swn}[2]{{{}_{{#1}}^{(#2)}}}
\newcommand{\nc}[2]{\newcommand{#1}{#2}}
\nc{\bmlp}{\mbox{\boldmath$\left(\right.$}}
\nc{\bmrp}{\mbox{\boldmath$\left.\right)$}}
\nc{\LAblp}{\mbox{\LARGE\boldmath$($}}
\nc{\LAbrp}{\mbox{\LARGE\boldmath$)$}}
\nc{\Lblp}{\mbox{\Large\boldmath$($}}
\nc{\Lbrp}{\mbox{\Large\boldmath$)$}}
\nc{\lblp}{\mbox{\large\boldmath$($}}
\nc{\lbrp}{\mbox{\large\boldmath$)$}}
\nc{\blp}{\mbox{\boldmath$($}}
\nc{\brp}{\mbox{\boldmath$)$}}
\nc{\LAlp}{\mbox{\LARGE $($}}
\nc{\LArp}{\mbox{\LARGE $)$}}
\nc{\Llp}{\mbox{\Large $($}}
\nc{\Lrp}{\mbox{\Large $)$}}
\nc{\llp}{\mbox{\large $($}}
\nc{\lrp}{\mbox{\large $)$}}
\nc{\lbc}{\mbox{\Large\boldmath$,$}}
\nc{\lc}{\mbox{\Large$,$}}
\nc{\Lall}{\mbox{\Large$\forall\;$}}
\nc{\Lexists}{\mbox{\Large$\exists\;$}}
\nc{\bc}{\mbox{\boldmath$,$}}

\nc{\LAbls}{\mbox{\LARGE\boldmath$[$}}
\nc{\LAbrs}{\mbox{\LARGE\boldmath$]$}}
\nc{\Lbls}{\mbox{\Large\boldmath$[$}}
\nc{\Lbrs}{\mbox{\Large\boldmath$]$}}
\nc{\lbls}{\mbox{\large\boldmath$[$}}
\nc{\lbrs}{\mbox{\large\boldmath$]$}}
\nc{\bls}{\mbox{\boldmath$[$}}
\nc{\brs}{\mbox{\boldmath$]$}}
\nc{\LAls}{\mbox{\LARGE $[$}}
\nc{\LArs}{\mbox{\LARGE $]$}}
\nc{\Lls}{\mbox{\Large $[$}}
\nc{\Lrs}{\mbox{\Large $]$}}
\nc{\lls}{\mbox{\large $[$}}
\nc{\lrs}{\mbox{\large $]$}}
\newcommand{\upset}[1]{{\uparrow\!\! #1}}

\newcommand{\oN}{{\ord{N}}}

\title[Quantum projective space  from Toeplitz cubes]
{{\Large Quantum projective space} \\ \vspace{2.5mm} {\Large from Toeplitz cubes}}

\author{Piotr~M.~Hajac}
\address{Instytut Matematyczny,
Polska Akademia Nauk,
ul.~\'Sniadeckich 8, Warszawa, 00-956 Poland\\
Katedra Metod Matematycznych Fizyki,
Uniwersytet Warszawski,
ul. Ho\.za 74, Warszawa, 00-682 Poland}
\email{http://www.impan.pl/\~{}pmh}

\author{Atabey~Kaygun} \address{Department of Mathematics and Computer
  Science, Bah\c{c}e\c{s}ehir University, \c{C}\i ra\u{g}an Cad.,
  Be\c{s}ikta\c{s} 34353 Istanbul, Turkey}
\email{atabey.kaygun@bahcesehir.edu.tr}

\author{Bartosz~Zieli\'nski}
\address{Instytut Matematyczny, Polska Akademia Nauk,
ul.~\'Sniadeckich 8, Warszawa, 00-956 Poland\\
Department of Theoretical Physics II, University of \L{}\'od\'z,
Pomorska 149/153 90-236 \L{}\'od\'z, Poland}
\email{bzielinski@uni.lodz.pl}

\begin{document}
\baselineskip=15.9pt

\begin{abstract}
From $N$-tensor powers of the Toeplitz algebra, we
  construct  a multipullback C*-algebra that is a noncommutative
   deformation of  the complex projective space $\cpN$. Using
   Birkhoff's Representation Theorem, we prove that the lattice of
   kernels of the canonical projections on components of the multipullback C*-algebra is free. This shows that our deformation
   preserves the freeness of  the lattice of subsets
   generated by the affine covering of the complex projective space.
\end{abstract}

\vspace*{-15mm}\maketitle
{\it\large\hfil Dedicated to Henri Moscovici on the occasion of his
65th birthday.\hfil}
\smallskip

\section*{Introduction}

\subsection{Motivation}

The procedure of decomposing complicated spaces into the union of
simple subsets and
applying Mayer-Vietoris type arguments to understand thus decomposed spaces
from the gluing data of simple pieces is commonly used in mathematics.
Manifolds without boundary  fit particularly well this piecewise
approach because they are defined as spaces that are locally
diffeomorphic to $\mathbb{R}^n$.
Thus a manifold appears assembled
from standard pieces by the gluing data. The standard pieces
are contractible --- they are homeomorphic to a ball. They encode
only the dimension of a manifold. All the
rest, topological properties of the manifold included,
are described by the gluing data.

Recall that to study topological spaces, one typically uses open coverings.
They are, however, hard to describe in purely C*-algebraic terms. On the other
hand, the Gelfand transform  turns closed coverings of a compact
Hausdorff space $X$ into an appropriate
 set of surjections from the C*-algebra $C(X)$ of
continuous functions on $X$ onto other C*-algebras.
 Hence it is easier and more natural to consider closed coverings
if one wants a noncommutative generalisation in terms of C*-algebras.
 More specifically, one can
define a
covering of a quantum space to be a family of C*-algebra surjections whose
 kernels intersect to zero (cover the whole space).
 We refer to \cite{HKZ} and references therein for a more in-depth
 discussion of this issue.

The aim of this article
is to explore the method of constructing noncommutative
deformations of manifolds by deforming the standard pieces.
This method is an alternative to the global deformation methods.
Thus it is expected to yield new examples or provide a new perspective
on already known cases. Our deformation of the complex projective spaces
is related to but different from a much studied quantum-group example.
(See the last section for details.) By construction, it is particularly suited
for developing and testing a definition of the fiber-product of spectral
triples that should  describe
 a gluing of smooth (noncommutative) geometries along
their boundaries.

Finally, let us note that complex projective spaces are
 topologically interesting
manifolds equipped with non-trivial tautological line bundles.
It seems very plausible that our Toeplitz projective spaces enjoy the same
 type of topological non-triviality and lead to interesting K-theoretic
 computations. They should also lead to non-crossed product
 $U(1)$-C*-algebras (non-trivial $U(1)$ quantum principal bundles).
 Using index theory, this has already been achieved
 for $N=1$, i.e., for the mirror quantum
 sphere~\cite{HajacMatthesSzymanski:Mirror}.

\subsection{Main result}

Our  main result concerns a new noncommutative deformation of the
complex projective space  and the  lattice generated by its
affine covering.
The guiding principle of our deformation is to
preserve the gluing data of this manifold while deforming the standard
pieces. We refine the affine covering of a complex projective space
to the Cartesian powers of unit discs, and replace the algebra of
continuous functions on the disk by the Toeplitz algebra commonly
regarded as the algebra of a quantum
disc~\cite{KlimekLesniewski:QuantumDisk}. The main point here is that
we preserve the freeness property enjoyed by the lattice generated
by the affine covering of the complex projective space:

\noindent{\bf Theorem~\ref{mainfree}.}
{\em
Let $C(\qcpN)\subset\prod_{i=0}^N\Tplz^{\otimes N}$ be the
C*-algebra of the Toeplitz quantum projective space, and let
  $\pi_i\colon C(\qcpN)\rightarrow\Tplz^{\otimes N}$,
 $i\in\oN$,
be the family of restrictions of the canonical projections onto
the components.  Then the
family of ideals $\{\ker\pi_i\}_{i\in\{0,\ldots,N\}}$ generates a free
distributive lattice.
}

\subsection{Notation and conventions}

In this article, the tensor product means the C*-completed  tensor
product.  Accordingly, we use the
Heynemann-Sweedler notation for the
completed tensor product. Since all C*-algebras that we tensor
are nuclear, this
 completion is unique. Therefore, it is also maximal, which guarantees
 the flatness of the completed tensor product. We use this property
in our arguments. Since the subsets $\{0,\ldots,N\}\subset\mN$, $N\in\mN$,
occur in abundance throughout this paper, for the sake brevity we use
the notation
\begin{equation}
\oN:=\{0,\ldots,N\}.
\end{equation}


\section{Preliminaries}

\subsection{Free lattices and Birkhoff's Representation Theorem}
\label{ProjectiveSpacesOverZ2}

We first recall definitions and simple facts about ordered sets and
lattices to fix terminology and notation.
Our main references on the subject are
\cite{Birkhoff:Lattices, BS:UniversalAlgebra,
  Stanley:EnumerativeCombinatoricsVol1}.

A set $P$ together with a binary relation $\leq$ is called {\em a
  partially ordered set}, or {\em a poset} in short, if the relation
$\leq$ is (i) reflexive, i.e., $p\leq p$ for any $p\in P$, (ii)
transitive, i.e., $p\leq q$ and $q\leq r$ implies $p\leq r$ for any
$p,q,r\in P$, and (iii) anti-symmetric, i.e., $p\leq q$ and $q\leq p$
implies $p=q$ for any $p,q\in P$.  If only the conditions (i)-(ii) are
satisfied we call $\leq$ a {\em preorder}.  For every preordered set
$(P,\leq)$ there is an opposite preordered set $(P,\leq)^\op$ given by
$P=P^\op$ and $p\leq^\op q$ if and only if $q\leq p$ for any $p,q\in
P$.

A poset $(P,\leq)$ is called {\em a semi-lattice} if for every $p,q\in
P$ there exists an element $p\vee q$ such that (i) $p\leq p\vee q$,
(ii) $q\leq p\vee q$, and (iii) if $r\in P$ is an element which
satisfies $p\leq r$ and $q\leq r$ then $p\vee q\leq r$.  The binary
operation $\vee$ is called {\em the join}.  A poset is called {\em a
  lattice} if both $(P,\leq)$ and $(P,\leq)^\op$ are semi-lattices.
The join operation in $P^\op$ is called {\em the meet}, and
traditionally denoted by $\wedge$.  One can equivalently define a
lattice $P$ as a set with two binary associative commutative and
idempotent operations $\vee$ and $\wedge$.  These operations satisfy
two absorption laws: $p = p\vee(p\wedge q)$ and $p= p\wedge(p\vee q)$
for any $p,q\in P$.  A lattice $(P,\vee,\wedge)$ is called {\em
  distributive} if one has $p\wedge(q\vee r) = (p\wedge q)\vee(p\wedge
r)$ for any $p,q,r\in P$.  Note that one can prove that the
distributivity of meet over join we have here is equivalent to the
distributivity of join over meet.

  Let $(P,\leq)$ be a preordered set, and let $\upset p := \{q\in P|\
  p\leq q\}$ for any $p\in P$. As a natural extension of notation, we
  define $\upset U:=\bigcup_{p\in U}\upset p$ for any $U\subseteq P$.
  The subsets $U\subseteq P$ that satisfy $U=\upset U$
  are called {\em upper sets} or {\em dual order ideals}.

Next, let $\Lambda$ be any lattice.
An element $c\in\Lambda$ is
called {\em meet irreducible} if 
$$
(i)\quad c=a\wedge b \quad\Rightarrow\quad
(c=a \quad\text{or}\quad c=b),\qquad
(ii)\quad \exists\; \lambda\in\Lambda:\; \lambda\not\leq c.
$$
The set of meet irreducible elements of the lattice $\Lambda$ is denoted
$\Mirred(\Lambda)$.  The {\em join irreducibles} $\Jirred(\Lambda)$
are defined dually.
 {\em Birkhoff's Representation Theorem}
\cite{Birkhoff:Rings} states that, if $\Lambda$ is a {\em finite
distributive lattice}, then the map
\begin{equation}
  \Lambda\ni a\longmapsto \{x\in\Mirred(\Lambda)\;|\;x\geq a\}=
  \Mirred(\Lambda)\cap\upset{a}\in \Up(\Mirred(\Lambda))
\end{equation}
assigning to $a$ the set of  meet
irreducible elements ${}\geq a$ is a lattice
isomorphism between $\Lambda$ and the lattice $\Up(\Mirred(\Lambda))$ of
upper sets of meet-irreducible elements of $\Lambda$ with $\cap$ and $\cup$
as its join and meet, respectively. We refer to this isomorphism as the
Birkhoff transform. Let us observe that it is analogous to the Gelfand
 transform: every finite
distributive lattice is the lattice of upper sets of a certain poset
just as every unital commutative C*-algebra is the algebra of
continuous functions on a certain
compact Hausdorff space.

As an immediate consequence of Birkhoff's Representation Theorem, one
sees that two finite distributive lattices are isomorphic if and only
if their posets of meet irreducibles are isomorphic.  In particular,
 consider a free distributive lattice generated by
$\lambda_0,\ldots,\lambda_N$, i.e., a lattice enjoying the universal property
that it admits a lattice homomorphism into any distributive lattice generated
by $N+1$ elements. It is isomorphic to the lattice of  non-empty
upper sets of the set of  non-empty subsets of $\oN$ (e.g.,
see \cite[Sect.~2.2]{HKMZ:PiecewisePrincipalComoduleAlgebras}).
The elements of the form
$\bigvee_{i\in I}\lambda_i$, where
$\emptyset\neq I\subsetneq\oN$, are all meet irreducible and partially ordered
by
\begin{equation}
\label{freeirrposet}
\bigvee_{i\in I}\lambda_i\leq \bigvee_{j\in J}
\lambda_j\quad\text{if and only if}\quad
I\subseteq J,\quad \forall\;
I,J\neq\emptyset,\; I,J\subsetneq\oN.
\end{equation}
In particular, they are all distinct.
Secondly, all meet irreducible elements
must be of the form $\bigvee_{i\in I}\lambda_i$, where
$\emptyset\neq I\subsetneq\oN$.

The first property can be easily deduced
from the upper-set model of a finite free distributive lattice,
and the latter holds for any finite distributive lattice.
Indeed, suppose the contrary, i.e., that there
exists a meet-irreducible
element whose any presentation
$\bigvee_{a\in \alpha}\bigwedge_{i\in a}\lambda_i$ is such that there
is a set $a_0\in\alpha$ that contains
at least two elements. Now,
the finiteness allows us to apply induction, and the
distributivity combined with irreducibility allows us to make
the induction step yielding the desired contradiction.

Thus we conclude the following lemma:
\begin{lemma}\label{buenos}
 A  finitely generated
distributive lattice is free if the
joins of its generators\linebreak
$\left\{
\bigvee_{i\in I}\lambda_i\right\}_{\emptyset\neq I
\subsetneq\oN}$ are all
meet irreducible and satisfy~\eqref{freeirrposet}.
\end{lemma}

\subsection{Closed covering of {\boldmath$\cpN$}
as an example of a free lattice}
\label{classsectcp}~

In \cite{HKMZ:PiecewisePrincipalComoduleAlgebras},
a closed refinement of the affine covering of $\cpN$ was constructed
as an example of a finite closed covering of a compact Hausdorff space.
Let us recall this construction.
The elements of this covering are given by:
\begin{equation}
V_i:=\{[x_0:\ldots:x_N]\;|\ |x_i|=\max\{|x_0|,\ldots,|x_N|\}\},\quad
i\in\ord{N}.
\end{equation}
It is easy to see that the family $\{V_i\}_{i\in\ord{N}}$ of closed
subsets of $\cpN$ is a covering of $\cpN$, i.e., $\bigcup_iV_i=\cpN$.
This covering is interesting because of its following property:
\begin{prop}
  The distributive
lattice $\Lambda$ generated by the subsets $V_i\subset\cpN$, $i\in\ord{N}$,
 is free.
\end{prop}
\begin{proof}
We prove the freenes of $\Lambda$ by showing that $\Lambda$ is isomorphic as a lattice with the lattice $\Upsilon$ of
non-empty upper sets of non-empty subsets of $\ord{N}$, which is a well known model of a free distributive lattice
(see, e.g.,~\cite{Birkhoff:Lattices}). For brevity, if $\emptyset\neq
a\subseteq\ord{N}$, we write
$V_a:=\bigcap_{i\in a}V_i$. Note that any $V\in\Lambda$ can be written as
 $V=\bigcup_{a\in A}V_a$ for some set $A$ of subsets of~$\oN$.
 We want to show that the following two maps are
 mutually inverse lattice isomorphisms:
\begin{gather}
R:\Upsilon\ni X\longmapsto \bigcup_{a\in X}V_a\in\Lambda,\quad
L:\Lambda\ni V\longmapsto \{a\in\ord{N}\;|\;V_a\subseteq V\}\in\Upsilon.
\end{gather}
For a proof that $R$ is a lattice map see,
e.g.,~\cite[Sect.~2.2]{HKMZ:PiecewisePrincipalComoduleAlgebras}.
The equality $R\circ L=\id$ is immediate.

The other equality $L\circ R=\id$ can be proven as follows.
Put
$
z_a:=[x_0:\ldots:x_N]\in\cpN,
$
 where $|x_i|=\max\{|x_0|,\ldots,|x_N|\}\Leftrightarrow i\in a$.
 Then one can easily see that
$
z_a\in V_b\Leftrightarrow b\subseteq a
$.
Hence $z_a\in V\Leftrightarrow V_a\subseteq V$, for all $V\in\Lambda$.
Therefore,   $a\in L(R(X))$ if and only if $z_a\in R(X)$, for all $X\in\Upsilon$.
 Finally, using again the  property
  $
z_a\in V_b\Leftrightarrow b\subseteq a
$
and the fact that $X$ is an upper set, we see
that $z_a\in R(X)$ if and only if $a\in X$.
\end{proof}

Now we  use the covering $\{V_i\}_{i\in\ord{N}}$
to present  $\cpN$ as a multipushout,
and, consequently, its C*-algebra $C(\cpN)$ as a multipullback.
To this end, we first define a family of homeomorphisms:
\begin{gather}
  \psi_i:V_i\longrightarrow D^{\times N} := \underbrace{D\times\ldots\times D}_{N\ \text{times}},\nonumber\\
  [x_0:\ldots:x_N]\longmapsto \left(
    \frac{x_0}{x_i},\ldots,\frac{x_{i-1}}{x_i},\frac{x_{i+1}}{x_i},\ldots,\frac{x_N}{x_i}
  \right),
\end{gather}
for all $i\in\ord{N}$, from $V_i$ onto the Cartesian product of
$N$-copies of $1$-disk. The inverses of the maps $\psi_i$ are given explicitly
by
\begin{equation}
\psi_i^{-1}:D^{\times N}\ni(d_1,\ldots,d_N)\longmapsto
[d_1:\ldots:d_i:1:d_{i+1}:\ldots:d_N]\in\cpN.
\end{equation}
Pick indices $0\leq i<j\leq N$ and consider the following commutative
diagram:
\begin{equation}\label{classdiag}
\xymatrix{
 & & \cpN & & \\
 D^{\times N} \ar@{-->}[urr] &
V_i \ar[l]_-{\psi_i}\ar@{^{(}->}[ur] & &
V_j\ar[r]^-{\psi_j}\ar@{_{(}->}[ul] &
D^{\times N}\ar@{-->}[ull]\\
D^{\times {j-1}}\times S\times D^{\times {N-j}} \ar@{^{(}->}[u] & &
V_i\cap V_j \ar@{_{(}->}[ul] \ar@{^{(}->}[ur] \ar[ll]_-{\psi_{ij}} \ar[rr]^-{\psi_{ji}} & &
D^{\times i}\times S\times D^{\times {N-i-1}}. \ar@<2ex>@{_{(}->}[u]
}
\end{equation}
Here, for
\begin{equation}
k=\begin{cases}
    n & \text{ if } m<n,\\
  n+1 & \text{ if } m>n,
\end{cases}
\end{equation}
we have
\begin{equation}
\psi_{mn}:=\left.\psi_m\right|_{V_m\cap V_n}:
V_m\cap V_n\longrightarrow D^{\times {k-1}}\times S\times D^{\times {N-k}}.
\end{equation}
In other words, counting from $1$, the $1$-circle $S$ appears on the
$k$th position among disks.  It follows immediately from the
definition of $\psi_i$ that the maps
\begin{equation}
\Upsilon_{ij}:=\psi_{ji}\circ\psi_{ij}^{-1}:D^{\times {j-1}}\times S
\times D^{\times {N-j}}
\longrightarrow D^{\times i}\times S\times D^{\times {N-i-1}},\quad
i<j,
\end{equation}
can be  explicitly written as
\begin{multline}\label{PsiExplicitClass}
\Upsilon_{ij}(d_1,\ldots,d_{j-1},s,d_{j+1},\ldots,d_N) = \\
(s^{-1}d_1,\ldots,s^{-1}d_i,s^{-1},s^{-1}d_{i+1},\ldots,
s^{-1}d_{j-1},s^{-1}d_{j+1},\ldots,s^{-1}d_N).
\end{multline}
One sees from Diagram~\eqref{classdiag} that $\cpN$ is homeomorphic to
the disjoint union $\bigsqcup_{i=0}^N D^{\times N}_i$ of $(N+1)$-copies
of $D^{\times N}$ divided by the identifications prescribed by
the the following diagrams indexed by $i,j\in\ord{N}$, $i<j$,
\begin{equation}\label{colimcpn}
\xymatrix{
D^{\times N}_i &  & \!\!\!\!\!\!D^{\times N}_j\\
D^{\times {j-1}}\times S\times D^{\times {N-j}} \ar@{_{(}->}[u] \ar[rr]^
{\Upsilon_{ij}}  & &
D^{\times i}\times S\times D^{\times {N-i-1}}. \ar@<2ex>@{_{(}->}[u]
}
\end{equation}

Consequently, one sees that the C*-algebra $C(\cpN)$ of continuous
functions on $\cpN$ is isomorphic
 with the subalgebra of  $\prod_{i=0}^N C(D)^{\otimes N}_i$
defined by the compatibility conditions given by the diagrams dual to
the diagrams~\eqref{colimcpn}:
\begin{equation}\label{colimcpndu}
\xymatrix{
C(D)^{\otimes N}_i \ar@{->>}[d]&  & \!\!\!\!\!\!C(D)^{\otimes N}_j
\ar@<-2ex>@{->>}[d]\\
C(D)^{\otimes {j-1}}\otimes C(S)\otimes C(D)^{\times {N-j}}    & &
\ar[ll]_{\Upsilon_{ij}^*}
C(D)^{\otimes i}\otimes C(S)\otimes C(D)^{\otimes {N-i-1}}.
}
\end{equation}

\section{The multipullback C*-algebra of $\qcpN$}

As a starting point for our noncommutative deformation of a
complex projective space, we take the diagrams~\eqref{colimcpndu}
from
Section~\ref{classsectcp} and replace the
algebra $C(D)$ of continuous functions
on the unit disk by the Toeplitz algebra $\Tplz$  considered
as the algebra of continuous functions on a quantum
disk~\cite{KlimekLesniewski:QuantumDisk}.
Recall that the
Toeplitz algebra is the universal
C*-algebra generated by $z$ and $z^*$
satisfying
$z^\ast z=1$.
There is a well-known short exact sequence of C*-algebras
\begin{equation}
0\longrightarrow \Cpct\longrightarrow\Tplz
\stackrel{\sigma}{\longrightarrow} C(S^1)
\longrightarrow 0.
\end{equation}
Here $\sigma$ is the so-called symbol map defined by mapping $z$ to
the unitary generator $u$  of the algebra $C(S^1)$
of continuous functions
on a circle.  Note that the kernel
of the symbol map is the algebra $\Cpct$ of compact
operators.

Viewing $S^1$ as the unitary group $U(1)$, we obtain
 a compact quantum
group structure on the algebra~$C(S^1)$. Here the antipode is
 determined by $S(u)=u^{-1}$, the counit
by $\varepsilon(u)=1$, and finally the comultiplication
by $\Delta(u)=u\otimes u$. Using this Hopf-algebraic terminology
on the C*-level makes sense due to the commutativity of~$C(S^1)$.
The coaction of $C(S^1)$ on $\Tplz$ comes from the gauge action of
  $U(1)$ on $\Tplz$  that rescales $z$
by the elements of $U(1)$, i.e., $z \mapsto \lambda z$. Explicitly,
we have:
\begin{equation}
\label{gaugecoact}
\rho:\Tplz\longrightarrow\Tplz\otimes C(S^1)\cong C(S^1,\Tplz),
\quad \rho(z):=z\otimes u,\quad \rho(z)(\lambda)=\lambda z,
\quad \rho(t)=:t\sw{0}\otimes t\sw{1}.
\end{equation}
Next, we employ  the multiplication map $m$ of $C(S^1)$ and
the flip map
\begin{equation}
  \CS\otimes\Tpn{n}\ni f\otimes t_1\otimes\cdots\otimes t_n
\stackrel{\tau_n}{\longmapsto}
  t_1\otimes\cdots\otimes t_n\otimes f\in\Tpn{n}\otimes\CS
\end{equation}
 to extend $\rho$
to the diagonal
coaction
$
  \rho_n:\Tpn{n}\longrightarrow\Tpn{n}\otimes\CS
$
defined inductively by
\begin{equation}
  \rho_1=\rho,\quad
  \rho_{n+1}=(\id_{\Tpn{n+1}}\otimes m)\circ(\id_\Tplz\otimes\tau_n\otimes\id_{\CS})\circ
  (\rho\otimes\rho_n).
\end{equation}

Furthermore,
for all $0\leq i<j \leq N$, we define an isomorphism $\Psi_{ij}$
\begin{equation}
  \chi_j\circ\Psi\circ\chi^{-1}_{i+1}:\;
\Tpn{i}\otimes\CS\otimes\Tpn{N-i-1}\stackrel{\Psi_{ij}}{\longrightarrow}
\Tpn{j-1}\otimes\CS
\otimes\Tpn{N-j}\,.
\end{equation}
Here $\chi_j$ is given by
\begin{equation}
 \id_{\Tpn{j-1}}\otimes\tau_{N-j}^{-1}:\Tpn{N-1}\otimes\CS
\stackrel{\chi_j}{\longrightarrow}\Tpn{j-1}\otimes\CS\otimes\Tpn{N-j}
\end{equation}
and $\Psi$ by
\begin{gather}
 (\id_{\Tpn{N-1}}\otimes(S\circ m))\circ(\rho_{N-1}\otimes\id_{\CS}):\;
\Tpn{N-1}\otimes\CS\stackrel{\Psi}{\longrightarrow}\Tpn{N-1}\otimes\CS.
\end{gather}
Before proceeding further, let us prove the unipotent property
of $\Psi$, which we shall need later on.
\begin{lemma}\label{unipotent}
 $\Psi\circ\Psi=\id_{\Tpn{N-1}\otimes\CS}$\,.
\end{lemma}
\begin{proof}
 For any $\bigotimes_{1\leq
  i<N}t_i\otimes h\in \Tplz^{\otimes N}\otimes\CS$, we compute:
\begin{align}\nonumber
(\Psi\circ\Psi)\left(\bigotimes_{1\leq i<N}\!\!t_i\otimes h\right)
&=\Psi\left(\bigotimes_{1\leq i<N}\!\!t\swn{i}{0}\otimes S(\prod_{1\leq i<N}\!\!t\swn{i}{1}h)\right)\\ \nonumber
&=\bigotimes_{1\leq i<N}\!\!t\swn{i}{0}\otimes S\left((\prod_{1\leq i<N}\!\!t\swn{i}{1})
S(\prod_{1\leq j<N}\!\!t\swn{j}{2}h)\right)\\ \nonumber
&=\bigotimes_{1\leq i<N}\!\!t\swn{i}{0}\otimes S\left((\prod_{1\leq i<N}\!\!(t\swn{i}{1})S(t\swn{i}{2}))S(h)\right)\\
&=\bigotimes_{1\leq i<N}\!\!t_i\otimes h.
\end{align}
\vspace*{-9mm}

\end{proof}

Finally, to justify our construction of a quantum complex projective
space,
 observe that
the map $\Psi_{ij}$ can be easily seen as an analogue of the pullback
of the
map $\Upsilon_{ij}$ of~\eqref{PsiExplicitClass}.
\begin{dfn}\label{qcpn}
We define the C*-algebra
$C(\qcpN)$ as the limit of the diagram:
\begin{equation*}\label{colimqcpn}
\xymatrix@R=2mm@C=0.15cm{
0 & \ldots & i & \ldots & j & \ldots & N\\
\Tpn{N} & \ldots & \Tpn{N}\ar[dddd]_{\sigma_j}
& \ldots & \!\!\!\!\!\!\Tpn{N}  \ar@<-2ex>[dddd]^{\sigma_{i+1}}
& \ldots & \Tpn{N}\\
{\ }\\
{\ }\\
{\ }\\
\ldots & \ldots & \Tpn{j-1}\otimes\CS\otimes \Tpn{N-j}
  & &
\Tpn{i}\otimes\CS\otimes \Tpn{N-i-1} \ar[ll]_{\Psi_{ij}}
& \ldots & \ldots\;.
}
\end{equation*}
Here we take all $i,j\in\oN$, $i<j$, and define
$\sigma_k:=\id_{\Tpn{k-1}}\otimes\sigma\otimes \id_{\Tpn{N-k}}$,
$k\in \{1,\ldots,N\}$. 
We call $\,\qcpN$ a {$\,$\em Toeplitz quantum complex projective
 space}.
\end{dfn}

Note that by definition $C(\qcpN)\subseteq
\prod_{i=0}^N\Tpn{N}$. We will denote the restrictions of the canonical
projections on the components by
\begin{equation}\label{canproj}
\pi_i:C(\qcpN)\longrightarrow\Tpn{N},\qquad \forall\;i\in\oN.
\end{equation}
 Since these maps are C*-homomorphisms, the lattice generated
by their kernels is automatically distributive. On the other hand,
 it follows from Lemma~\ref{surjprop} that any element in the
Toeplitz cube $\Tplz^{\otimes n}$  can be complemented into
a sequence that is an element of $C(\qcpN)$. This means that
the maps \eqref{canproj} are surjective. Hence
they form
 a covering of~$C(\qcpN)$.

The construction of $\qcpN$ is a generalization of the construction of
the mirror quantum sphere \cite[p.~734]{HajacMatthesSzymanski:Mirror},
i.e.,
$\mathbb{P}^1(\C{T})$ is the mirror quantum sphere:
\begin{equation}
  C(\B{P}^1(\C{T})):=\{(t_0,t_1)\in \Tplz\times \Tplz\;|
\;\sigma(t_0)=S(\sigma(t_1))\}.
\end{equation}
Removing $S$ from this definition
yields the C*-algebra of the
generic Podle\'s sphere \cite{Podles:Spheres}. The latter not only
is not isomorphic with  $C(\B{P}^1(\C{T}))$, but also is not
Morita equivalent to
$C(\B{P}^1(\C{T}))$~\cite[Prop.~2.3]{HajacMatthesSzymanski:Mirror}.
We conjecture that, by similar changes
in maps $\Psi_{ij}$,
we can create non-equivalent
quantum spaces also for $N>1$.

\section{The defining covering lattice of $\qcpN$ is free}

The goal of this
 section is to demonstrate that the distributive lattice of
ideals generated by the kernels
 $\ker\pi_i$ is free.  To this end, we will need
to know whether the tensor products $\Tplz^{\otimes N}$ of Toeplitz
algebras glue together to form $\qcpN$ in such a way that a partial
gluing can be always extended to a full space.  The following result
gives the sufficient conditions:
\begin{prop}\cite[Prop.~9]{CalowMatthes:Glueing} \label{MattProp}
 Let $\{B_i\}_{i\in\oN}$ and
  $\{B_{ij}\}_{i,j\in\oN,\,i\neq j}$ be two families of
C*-algebras
  such that
  $B_{ij}=B_{ji}$\,, and let $\{\pi^i_j:B_i\rightarrow B_{ij}\}_{ij}$
be a family of surjective C*-algebra maps. Also,
let $\pi_i:B\rightarrow B_i$, $i\in\oN$,
  be the restrictions to
$$
B:=\{(b_i)_i\in\mbox{$\prod_{i\in\oN}$}B_i\;|\;\pi^i_j(b_i)=\pi^j_i(b_j),\;\forall\,
i,j\in\oN,\;i\neq j\}
$$
 of the canonical projections.  Assume that, for all triples
of distinct indices $i,j,k\in\oN$,
\begin{enumerate}
\item $\pi^i_j(\ker\pi^i_k)=\pi^j_i(\ker\pi^j_k)$,
\vspace{1mm}
\item the isomorphisms $\pi^{ij}_k\colon
  B_i/(\ker\pi^i_j+\ker\pi^i_k)\longrightarrow
  B_{ij}/\pi^i_j(\ker\pi^i_k)$ defined as
  \begin{align*}
    &b_i+\ker\pi^i_j+\ker\pi^i_k\longmapsto \pi^i_j(b_i)+\pi^i_j(\ker\pi^i_k)
  \\
  \text{ satisfy }\quad  &(\pi^{ik}_j)^{-1}\circ\pi^{ki}_j=(\pi^{ij}_k)^{-1}\circ\pi^{ji}_k\circ(\pi^{jk}_i)^{-1}\circ\pi^{kj}_i.
  \end{align*}
\end{enumerate}\vspace*{-5mm}
\begin{align*}
\text{Then, }\;&\Lall(b_i)_{i\in I}\in\mbox{$\prod_{i\in I}$}\;
B_i,\;  I\subseteq\oN,
\text{ such that }
\pi^i_j(b_i)=\pi^j_i(b_j),\, \forall\; i,j\in I,\,i\neq j,
\\
&\Lexists
(c_i)_{i\in\oN}\in\mbox{$\prod_{i\in\oN}$}\; B_i : \:
\pi^i_j(c_i)=\pi^j_i(c_j),\,\forall\; i,j\in\oN,\,i\neq j, \text{ and }
c_i=b_i, \forall\; i\in I.
\end{align*}
\end{prop}

In the case of quantum projective spaces $\qcpN$,
 we can translate algebras and maps from Proposition~\ref{MattProp}
as follows:
\begin{gather}
B_i=\Tpn{N},\quad
B_{ij}=\Tpn{j-1}\otimes\CS\otimes \Tpn{N-j},\quad\text{where}\quad i<j,\\
\pi^i_j=\left\{
\begin{array}{ccc}
\sigma_j & \text{when} & i<j,\\
\Psi_{ji}\circ\sigma_{j+1} &\text{when} & i>j.
\end{array}
\right.
\end{gather}
It follows that
\begin{equation}
\ker\pi^i_j=
\begin{cases}
  \ker\sigma_j=\Tpn{j-1}\otimes\Cpct\otimes \Tpn{N-j} & \text{when\ }  i<j,\\
  \ker\sigma_{j+1}=\Tpn{j}\otimes\Cpct\otimes\Tpn{N-j-1} & \text{when\ }  i>j.
\end{cases}
\end{equation}
Since $\rho(\Cpct)\subseteq\Cpct\otimes\CS$ and $\Psi$ is an
isomorphism by Lemma~\ref{unipotent}, it follows that
\begin{equation}
\Psi(\Tpn{j-1}\otimes\Cpct\otimes \Tpn{N-j-1}\otimes\CS)=\Tpn{j-1}\otimes\Cpct\otimes \Tpn{N-j-1}\otimes\CS.
\end{equation}
Now we can formulate and prove the following:
\begin{lemma}\label{surjprop}
If $(b_i)_{i\in I}\in\prod_{i\in I\subseteq\oN}\Tpn{N}$
satisfies
$\pi^i_j(b_i)=\pi^j_i(b_j)$ for all $i,j\in I$, $i\neq j$,
then there exists an element
$b\in C(\qcpN)$ such that $\pi_i(b)=b_i$ for all $i\in I$.
\end{lemma}
\begin{proof}
  It is enough to check that the assumptions of
  Proposition~\ref{MattProp} are satisfied.  For the sake of brevity,
  we will omit the tensor symbols in the long formulas in what
  follows. We will also write $\C{S}$ instead of
  $\CS$.
  \renewcommand{\CS}{{\mathcal{S}}}\renewcommand{\Tpn}[1]{\Tplz^{#1}}
  Here we prove the first condition of Proposition~\ref{MattProp}:
\begin{enumerate}
\item
$\pi^j_i(\ker\pi^j_k)
=(\chi_j\circ\Psi\circ\chi_{i+1}^{-1}\circ\sigma_{i+1})(\ker\sigma_{k+1})\\
\phantom{\pi^j_i(\ker\pi^j_k)}=
(\chi_j\circ\Psi\circ\chi_{i+1}^{-1})(\Tpn{i}\CS\Tpn{k-i-1}\Cpct\Tpn{N-k-1})\\
\phantom{\pi^j_i(\ker\pi^j_k)}=
\chi_j(\Tpn{k-1}\Cpct\Tpn{N-k-1}\CS)\\
\phantom{\pi^j_i(\ker\pi^j_k)}=
\Tpn{k-1}\Cpct\Tpn{j-k-1}\CS\Tpn{N-j}\\
\phantom{\pi^j_i(\ker\pi^j_k)}=
\sigma_j(\ker\sigma_k)\\
\phantom{\pi^j_i(\ker\pi^j_k)}=\pi^i_j(\ker\pi^i_k),\quad
\text{when } i<k<j$.\\
\item $\pi^j_i(\ker\pi^j_k)=(\chi_j\circ\Psi\circ\chi_{i+1}^{-1}\circ\sigma_{i+1})(\ker\sigma_k)\\
\phantom{\pi^j_i(\ker\pi^j_k)}=
(\chi_j\circ\Psi\circ\chi_{i+1}^{-1})(\Tpn{i}\CS\Tpn{k-i-2}\Cpct\Tpn{N-k})\\
\phantom{\pi^j_i(\ker\pi^j_k)}
=\chi_j(\Tpn{k-2}\Cpct\Tpn{N-k} \CS)\\
\phantom{\pi^j_i(\ker\pi^j_k)}
=\Tpn{j-1}\CS\Tpn{k-j-1}\Cpct\Tpn{N-k}\\
\phantom{\pi^j_i(\ker\pi^j_k)}
=\sigma_j(\ker\sigma_k)\\
\phantom{\pi^j_i(\ker\pi^j_k)}
=\pi^i_j(\ker\pi^i_k),\quad
\text{when }  i<j<k$.\\
\item $\pi^j_i(\ker\pi^j_k)=(\chi_j\circ\Psi\circ\chi_{i+1}^{-1}\circ\sigma_{i+1})(\ker\sigma_{k+1})\\
    \phantom{\pi^j_i(\ker\pi^j_k)}
=(\chi_j\circ\Psi\circ\chi_{i+1}^{-1})(\Tpn{k}\Cpct\Tpn{i-k-1}\CS\Tpn{N-i-1})\\
\phantom{\pi^j_i(\ker\pi^j_k)}
=\chi_j(\Tpn{k}\Cpct\Tpn{N-k-2}\CS)\\
\phantom{\pi^j_i(\ker\pi^j_k)}
=\Tpn{k}\Cpct\Tpn{j-k-2}\CS\Tpn{N-j}\\
\phantom{\pi^j_i(\ker\pi^j_k)}
=\sigma_j(\ker\sigma_{k+1})\\
\phantom{\pi^j_i(\ker\pi^j_k)}
=\pi^i_j(\ker\pi^i_k),\quad \text{when } k<i<j$.
\end{enumerate}

For the second condition, note first that
for any multivalued map $f:B_j\rightarrow B_i$ we define the
 function
\begin{gather}
[f]^{ij}_k:B_j/(\ker\pi^j_i+\ker\pi^j_k)\longrightarrow
B_i/(\ker\pi^i_j+\ker\pi^i_k),\nonumber\\
b_j+\ker\pi^j_i+\ker\pi^j_k\longmapsto f(b_j)+\ker\pi^i_j+\ker\pi^i_k,
\label{badlydefined}
\end{gather}
whenever the assignement~\eqref{badlydefined} is unique. In particular,
since the condition (1) of Proposition~\ref{MattProp} is fulfilled,
we can write
the map
$\phi^{ij}_k:=(\pi^{ij}_k)^{-1}\circ\pi^{ji}_k$
as $[(\pi^i_j)^{-1}\circ\pi^j_i]^{ij}_k$.
Explicitly, in  our case, this map
reads:
\begin{equation}
\phi^{ij}_k=\left\{
\begin{array}{ccc}
[\sigma_j^{-1}\circ\chi_j\circ\Psi\circ\chi_{i+1}^{-1}\circ\sigma_{i+1}]^{ij}_k & \text{when} & i<j,\\
{}[\sigma_{j+1}^{-1}\circ\chi_{j+1}\circ\Psi\circ\chi_{i}^{-1}\circ\sigma_{i}]^{ij}_k & \text{when} & i>j.
\end{array}
\right.
\end{equation}
We need to prove that
\begin{equation}
\label{cocyclec}
\phi^{ij}_k=\phi^{ik}_j\circ\phi^{kj}_i,\quad\text{for all distinct indices\ }
 i,j,k.
\end{equation}
Since $(\phi^{ij}_k)^{-1}=\phi^{ji}_k$ and, for any invertible elements $g,h,k$, the equality
$k=gh$ can be written as $h=g^{-1}k$, etc.,
one can readily see that it is enough to limit ourselves to the case
when $i<k<j$. 
Next, let
us denote the class of  $(t_1\otimes\cdots\otimes
t_N)=\bigotimes_{1\leq n\leq N}t_n\in\Tplz^{\otimes N}$ in
$\Tplz^{\otimes N}\!\!/(\ker\pi^j_i+\ker\pi^j_k)$ by
$[\bigotimes_{1\leq n\leq N}t_n]^{j}_{ik}$. Then, using
the Heynemann-Sweedler
notation for completed tensor products, we compute:
\begin{align}\nonumber
\phi^{ij}_k\Llp\Lls\!\bigotimes_{1\leq n\leq N}\!\!t_n\Lrs^{j}_{ik}\Lrp
&=[\sigma_j^{-1}\circ\chi_j\circ\Psi\circ\chi_{i+1}^{-1}\circ\sigma_{i+1}]^{ij}_k
\Llp\Lls\!\bigotimes_{1\leq n\leq N}\!\!t_n\Lrs^{j}_{ik}\Lrp\\ \nonumber
&=\Lls(\sigma_j^{-1}\circ\chi_j\circ\Psi)\Llp\!\!\bigotimes_{1\leq n\leq N\atop n\neq i+1}\!\!t_n\otimes
\sigma(t_{i+1})\Lrp\Lrs^i_{jk}\\ \nonumber
&=\Lls(\sigma_j^{-1}\circ\chi_j)\Llp\!\!\bigotimes_{1\leq n\leq N\atop n\neq i+1}\!\!t\swn{n}{0}\otimes
S\Llp\sigma(t_{i+1})\!\!\prod_{1\leq m\leq N\atop m\neq i+1}\!\!t\swn{m}{1}
\Lrp\Lrp\Lrs^i_{jk}\\
&=\Lls\!\!\bigotimes_{1\leq n\leq j\atop n\neq i+1}\!\!t\swn{n}{0}\otimes
(\sigma^{-1}\circ S)\Llp
\sigma(t_{i+1})\!\!\prod_{1\leq m\leq N\atop m\neq i+1}\!\!t\swn{m}{1}
\Lrp
\otimes\!\!\!\!\bigotimes_{j+1\leq s\leq N}\!\!\!\!t\swn{s}{0}\Lrs^i_{jk}.
\end{align}
Applying the above formula twice (with the  non-dummy indices changed),
we obtain:
\begin{align}
\nonumber
{}&(\phi^{ik}_j\circ\phi^{kj}_i)\Llp\Lls\!\bigotimes_{1\leq n\leq N}\!\!t_n\Lrs^{j}_{ik}\Lrp
=\phi^{ik}_j\Llp\Lls\!\!\bigotimes_{1\leq n\leq j\atop n\neq k+1}\!\!t\swn{n}{0}\otimes
(\sigma^{-1}\circ S)\Llp\sigma(t_{k+1})\!\!\prod_{1\leq m\leq N\atop m\neq k+1}\!\!t\swn{m}{1}
\Lrp\otimes\!\!\!\!\bigotimes_{j+1\leq s\leq N}\!\!\!\!t\swn{s}{0}\Lrs^k_{ji}\Lrp\\
\nonumber
{}&=\Lls\!\!\bigotimes_{1\leq n\leq k\atop n\neq i+1}\!\!t\swn{n}{0}\sw{0}\otimes
(\sigma^{-1}\circ S)\Llp
\sigma(t\swn{i+1}{0})
\Lblp(\sigma^{-1}\circ S)\llp\sigma(t_{k+1})\!\!\prod_{1\leq m\leq N\atop m\neq k+1}\!\!t\swn{m}{1}\lrp\Lbrp\sw{1}
\!\!\prod_{1\leq w\leq N\atop {w\neq i+1\atop w\neq k+1}}\!\!t\swn{w}{0}\sw{1}
\Lrp\\
{}&\phantom{=([}\otimes\!\!\!\!\bigotimes_{k+2\leq r\leq j}\!\!\!\!t\swn{n}{0}\sw{0}\otimes
\Lblp(\sigma^{-1}\circ S)\llp\sigma(t_{k+1})\!\!\prod_{1\leq m\leq N\atop m\neq k+1}\!\!t\swn{m}{1}
\lrp\Lbrp\sw{0}\otimes\!\!\!\!\bigotimes_{j+1\leq s\leq N}\!\!\!\!t\swn{s}{0}\sw{0}\Lrs^i_{jk}.
\end{align}
Now, as $\sigma^{-1}:C(S^1)\rightarrow
\Tplz/\Cpct$ is colinear, $S$ is an anti-coalgebra map, and $\Delta$
is an algebra homomorphism, we can move the Heynemann-Sweedler
 indices inside
the bold parentheses:
\begin{align}\nonumber
{}&\Lls\bigotimes_{1\leq n\leq k\atop n\neq i+1}t\swn{n}{0}\sw{0}\otimes
(\sigma^{-1}\circ S)\Llp\sigma(t\swn{i+1}{0})
S\Lblp\sigma(t_{k+1})\sw{1}\!\!\prod_{1\leq m\leq N\atop m\neq k+1}\!\!t\swn{m}{1}\sw{1}\Lbrp
\!\!\prod_{1\leq w\leq N\atop {w\neq i+1\atop w\neq k+1}}\!\!t\swn{w}{0}\sw{1}
\Lrp\\
{}&\phantom{=([}\otimes\!\!\!\!\bigotimes_{k+2\leq r\leq j}\!\!\!\!t\swn{n}{0}\sw{0}\otimes
(\sigma^{-1}\circ S)\Llp\sigma(t_{k+1})\sw{2}\!\!\prod_{1\leq m\leq N\atop m\neq k+1}\!\!t\swn{m}{1}\sw{2}\Lrp
\otimes
\!\!\!\!\bigotimes_{j+1\leq s\leq N}\!\!\!\!t\swn{s}{0}\sw{0}\Lrs^i_{jk}.
\end{align}
Here we can renumber the Heynemann-Sweedler indices using the
 coassociativity of $\Delta$.
We can also use the anti-multiplicativity of $S$ to move it inside the
bold parentheses in the first line of the
above calculation.  Finally, we use the
commutativity of $C(S^1)$ in order to reshuffle
the argument of $\sigma^{-1}\circ S$ in the first line to obtain:
\begin{align}\nonumber
{}&\Lls\bigotimes_{1\leq n\leq k\atop n\neq i+1}t\swn{n}{0}\otimes
(\sigma^{-1}\circ S)\Llp\sigma(t\swn{i+1}{0})
S(t\swn{i+1}{1})S\llp\sigma(t_{k+1})\sw{1}\lrp
\!\!\prod_{1\leq w\leq N\atop {w\neq i+1\atop w\neq k+1}}\!\!\lblp t\swn{w}{1}S(t\swn{w}{2})\lbrp
\Lrp\\
{}&\phantom{=([}\otimes\!\!\!\!\bigotimes_{k+2\leq r\leq j}\!\!\!\!t\swn{n}{0}\otimes
(\sigma^{-1}\circ S)\Llp\sigma(t_{k+1})\sw{2}t\swn{i+1}{2}
\!\!\prod_{1\leq m\leq N\atop {m\neq k+1 \atop m\neq i+1}}\!\!t\swn{m}{3}\Lrp
\otimes
\!\!\!\!\bigotimes_{j+1\leq s\leq N}\!\!\!\!t\swn{s}{0}	\Lrs^i_{jk}.
\end{align}
We can simplify the expression in the bold parentheses in the first line
using $h\sw{1}S(h\sw{2})=\varepsilon(h)$ and
$\varepsilon(h\sw{1})h\sw{2}=h$. This results in:
\begin{align}\nonumber
{}&\Lls\bigotimes_{1\leq n\leq k\atop n\neq i+1}t\swn{n}{0}\otimes
(\sigma^{-1}\circ S)\Llp\sigma(t\swn{i+1}{0})
S(t\swn{i+1}{1})S\llp\sigma(t_{k+1})\sw{1}\lrp
\Lrp\\
{}&\phantom{=([}\otimes\!\!\!\!\bigotimes_{k+2\leq r\leq j}\!\!\!\!t\swn{n}{0}\otimes
(\sigma^{-1}\circ S)\Llp\sigma(t_{k+1})\sw{2}t\swn{i+1}{2}
\!\!\prod_{1\leq m\leq N\atop {m\neq k+1 \atop m\neq i+1}}\!\!t\swn{m}{1}\Lrp
\otimes
\!\!\!\!\bigotimes_{j+1\leq s\leq N}\!\!\!\!t\swn{s}{0}	\Lrs^i_{jk}.
\end{align}
By the colinearity of $\sigma$, we can substitute in the above expression
\begin{align}\nonumber
\sigma(t\swn{i+1}{0})\otimes t\swn{i+1}{1}&\mapsto \sigma(t_{i+1})\sw{1}\otimes\sigma(t_{i+1})\sw{2},\\
\sigma(t_{k+1})\sw{1}\otimes\sigma(t_{k+1})\sw{2}&\mapsto \sigma(t\swn{k+1}{0})\otimes t\swn{k+1}{1},
\end{align}
to derive:
\begin{align}\nonumber
{}&\Lls\bigotimes_{1\leq n\leq k\atop n\neq i+1}t\swn{n}{0}\otimes
(\sigma^{-1}\circ S)\Llp\sigma(t_{i+1})\sw{1}
S\llp\sigma(t_{i+1})\sw{2}\lrp S\llp\sigma(t\swn{k+1}{0})\lrp
\Lrp\\ 
{}&\phantom{=([}\otimes\!\!\!\!\bigotimes_{k+2\leq r\leq j}\!\!\!\!t\swn{n}{0}\otimes
(\sigma^{-1}\circ S)\Llp t\swn{k+1}{1} \sigma(t_{i+1})\sw{3}
\!\!\prod_{1\leq m\leq N\atop {m\neq k+1 \atop m\neq i+1}}\!\!t\swn{m}{1}\Lrp
\otimes
\!\!\!\!\bigotimes_{j+1\leq s\leq N}\!\!\!\!t\swn{s}{0}	\Lrs^i_{jk}.
\end{align}
Applying again the antipode and counit properties yields the desired
\begin{align}\nonumber
{}&\Lls\bigotimes_{1\leq n\leq k\atop n\neq i+1}t\swn{n}{0}\otimes
(\sigma^{-1}\circ S)\Llp S\llp\sigma(t\swn{k+1}{0})\lrp
\Lrp\\ \nonumber
{}&\phantom{=([}\otimes\!\!\!\!\bigotimes_{k+2\leq r\leq j}\!\!\!\!t\swn{n}{0}\otimes
(\sigma^{-1}\circ S)\Llp t\swn{k+1}{1} \sigma(t_{i+1})
\!\!\prod_{1\leq m\leq N\atop {m\neq k+1 \atop m\neq i+1}}\!\!t\swn{m}{1}\Lrp
\otimes
\!\!\!\!\bigotimes_{j+1\leq s\leq N}\!\!\!\!t\swn{s}{0}	\Lrs^i_{jk}\\
\nonumber
{}=&\Lls\bigotimes_{1\leq n\leq k\atop n\neq i+1}t\swn{n}{0}\otimes
t\swn{k+1}{0}\otimes
\!\!\!\!\bigotimes_{k+2\leq r\leq j}\!\!\!\!t\swn{n}{0}\otimes
(\sigma^{-1}\circ S)\Llp t\swn{k+1}{1} \sigma(t_{i+1})
\!\!\prod_{1\leq m\leq N\atop {m\neq k+1 \atop m\neq i+1}}\!\!t\swn{m}{1}\Lrp
\otimes
\!\!\!\!\bigotimes_{j+1\leq s\leq N}\!\!\!\!t\swn{s}{0}	\Lrs^i_{jk}\\
\nonumber
{}=&\Lls\bigotimes_{1\leq n\leq j\atop n\neq i+1}t\swn{n}{0}\otimes
(\sigma^{-1}\circ S)\Llp \sigma(t_{i+1})
\!\!\prod_{1\leq m\leq N\atop {m\neq i+1}}\!\!t\swn{m}{1}\Lrp
\otimes
\!\!\!\!\bigotimes_{j+1\leq s\leq N}\!\!\!\!t\swn{s}{0}	\Lrs^i_{jk}\\
{}=&\phi^{ij}_k\Llp\Lls\bigotimes_{1\leq n\leq N}t_n\Lrs^{j}_{ik}\Lrp.
\end{align}
\end{proof}

Our next step is to prove that the assumptions of Lemma~\ref{buenos} hold,
so that we can take advantage of Birkhoff's Representation Theorem
to conclude the freeness of the lattice generated by the ideals $\ker\pi_i$.

\begin{lemma}\label{posetisomfree}
 For all
  non-empty subsets $I,J\subseteq\oN$
  \begin{equation*}
   \bigcap_{i\in I}\ker\pi_i\supseteq \bigcap_{j\in J}\ker\pi_j\quad
    \text{if and only if}\quad I\subseteq J.
  \end{equation*}
\end{lemma}
\begin{proof}
The ``if"-implication is obvious. For the ``only if"-implication,
 take $0\neq x\in\Cpct^{\otimes N}$ and, for any non-empty
  $I\subseteq\oN$, define
\begin{equation}
  x_I:=(x_i)_{i\in\oN}\in \bigcap_{i\in I}\ker\pi_i,\quad \text{where}\quad
  x_i:= \begin{cases}
    x & \text{if } i\notin I,\\
    0 & \text{if } i\in I.
  \end{cases}
\end{equation}
Let $I,J\subseteq\oN$ be non-empty,
and assume that $I\setminus J$ is
non-empty.  Then it follows that
\begin{equation}
x_J\in  \left(\bigcap_{j\in J}\ker\pi_j\right)
\setminus\left(\bigcap_{i\in I}\ker\pi_i\right)\neq\emptyset.
\end{equation}
This means
that $\bigcap_{j\in J}\ker\pi_j\not\subseteq\bigcap_{i\in I}\ker\pi_i$,
as desired. It follows that $\bigcap_{i\in I}\ker\pi_i$ are all distinct.
\end{proof}

\begin{lemma}\label{meetirredl}
  The ideals $\bigcap_{i\in I}\ker\pi_i$ are all meet (sum)
  irreducible for any $\emptyset\neq I\subsetneq\oN$.
\end{lemma}
\begin{proof}
We proceed by contradiction. Suppose that $\bigcap_{i\in I}\ker\pi_i$
is not meet irreducible for some
$\emptyset\neq I\subsetneq\oN$.
By Lemma~\ref{posetisomfree},  $\bigcap_{i\in I}\ker\pi_i\neq \{0\}$
because $I\neq \oN$. Hence
there exist ideals
\begin{equation}
a_\mu=\sum_{J\in\mathcal{J}_\mu}\bigcap_{j\in J}\ker\pi_j,\quad
\mathcal{J}_\mu\subseteq \bs{2}^{\oN},
\quad \mu\in\{1,2\},
\end{equation}
such that
\begin{equation}
\bigcap_{i\in I}\ker\pi_i=a_1+a_2,\quad\text{and}\quad
 a_1,a_2\neq \bigcap_{i\in I}\ker\pi_i.
\end{equation}
In particular, $a_\mu\subseteq\bigcap_{i\in I}\ker\pi_i$,
$\mu\in\{1,2\}$.
On the other hand, if $I\in\mathcal{J}_\mu$, then
 $a_\mu\supseteq\bigcap_{i\in I}\ker\pi_i$. Hence
$a_\mu=\bigcap_{i\in I}\ker\pi_i$, contrary to our assumption. It follows
that, if $\bigcap_{i\in I}\ker\pi_i$ is not meet irreducible, then
\begin{equation}\label{assertl}
  \bigcap_{i\in I}\ker\pi_i=\sum_{J\in\mathcal{J}}\bigcap_{j\in J}\ker
\pi_j\,,\quad\text{for some}\quad
 \mathcal{J}\subseteq \bs{2}^{\oN}\setminus \{I\}.
\end{equation}

Suppose next that $I\setminus J_0$ is non-empty for some
$J_0\in\mathcal{J}$, and let $k\in I\setminus J_0$.
Then
\begin{equation}
  \{0\}=\pi_k\left(\bigcap_{i\in I}\ker\pi_i\right)
=\pi_k\left(\sum_{J\in\mathcal{J}}\bigcap_{j\in J}\ker\pi_j\right)
  \supseteq \pi_k\left(\bigcap_{j\in J_0}\ker\pi_j\right).
\end{equation}
However, from Lemma~\ref{posetisomfree} we see that $(\bigcap_{j\in
  J_0}\ker\pi_j)\setminus\ker\pi_k$ is non-empty.  Hence
$\pi_k(\bigcap_{j\in J_0}\ker\pi_j)$ is not $\{0\}$, and we have a
contradiction. It follows that for all $J_0\in\mathcal{J}$ the set
$I\setminus J_0$ is empty, i.e.,
$\forall\;J_0\in\mathcal{J}:\;I\subsetneq J_0$.

Finally, let $m\in\oN\setminus I$, and let
\begin{equation}
  T_m^I:=t_1\otimes\cdots\otimes t_N\,,
\quad\text{where}\quad 0\neq t_n\in
\left\{\begin{array}{ccc}
\Cpct & \text{if} &
m<n \in I \text{\ or\ } m> n-1 \in I, \\
\Tplz\setminus\Cpct & \text{if} &
m<n\notin I \text{\ or\ } m>n-1 \notin I.
\end{array}\right.
\end{equation}
Note that $\pi^m_k(T_m^I)=0$ if and only if $k\in I$. Hence,
by Lemma~\ref{surjprop},  there exists
$p_m\in\pi_m^{-1}(T_m^I)\cap\bigcap_{i\in I}\ker\pi_i$. Next, we
define
\begin{equation}
\sigma^m_I:=f_1\otimes\cdots\otimes f_N,\quad\text{where}\quad f_n:=
\left\{
\begin{array}{ccc}
\id_\Tplz &\text{if}&
m<n \in I \text{\ or\ } m> n-1 \in I, \\
\sigma & \text{if}&
m<n\notin I \text{\ or\ } m>n-1 \notin I,
\end{array}
\right.
\end{equation}
so that $\sigma^m_I(\pi_m(p_m))\neq 0$. On the other hand,
by our assumption~\eqref{assertl}, and the property that $J_0\supsetneq
I$ for all $J_0\in\mathcal{J}$, we have
\begin{equation}\label{contrad}
0\neq
p_m\in\bigcap_{i\in I}\ker\pi_i\subseteq \sum_{J\supsetneq I}\bigcap_{j
\in J}\ker\pi_j.
\end{equation}
Furthermore, for any $x\in C(\qcpN)$
\begin{equation}\label{quad}
\sigma^m_I(\pi_m(x))=0 \quad\text{if}\quad \pi^m_k(\pi_m(x))=0
\quad\text{for some}\quad k\notin I.
\end{equation}
Now, for any $J\supsetneq I$, we choose $k_J\in J\setminus I$, so that
\begin{equation}
\pi^m_{k_J}\Llp\pi_m\Llp\bigcap_{j\in J\supsetneq I}
\ker\pi_j\Lrp\Lrp
\subseteq\pi^{k_J}_m(\pi_{k_J}(\ker\pi_{k_J}))=\{0\}.
\end{equation}
Combining this with \eqref{quad}, we obtain
$\sigma^m_I(\pi_m(\bigcap_{j\in J\supsetneq I}\ker\pi_j))=\{0\}$
for all $J\supsetneq I$.
Consequently, $\sigma^m_I(\pi_m(\sum_{J\supsetneq I}\bigcap_{j\in
  J}\ker\pi_j))=\{0\}$, which contradicts \eqref{contrad}, and ends
the proof.
\end{proof}

  Summarizing,
  Lemma~\ref{posetisomfree} and Lemma~\ref{meetirredl}
combined with Lemma~\ref{buenos} yield the  main result of this paper:
\begin{thm}\label{mainfree}
  Let $C(\qcpN)\subset\prod_{i=0}^N\Tplz^{\otimes N}$ be the
C*-algebra of the Toeplitz quantum projective space, defined as the
 limit  of Diagram~\eqref{colimqcpn}, and let
  \begin{equation*}
    \pi_i\colon C(\qcpN)\longrightarrow\Tplz^{\otimes N},\quad
 i\in\oN,
\end{equation*}
be the family of restrictions of the canonical projections onto
the components.  Then the
family of ideals $\{\ker\pi_i\}_{i\in\oN}$ generates a free
distributive lattice.
\end{thm}

\section{Other quantum projective spaces}

Let us first compare our construction of quantum
complex projective spaces with the construction  coming from
quantum groups. Then we complete this comparison by describing other
noncommutative versions of complex projective spaces that we found
in the literature.

\subsection{Noncommutative projective spaces as homogeneous spaces
  over quantum groups}

Complex projective spaces are fundamental examples of compact manifolds
without boundary. They can be viewed as the quotient spaces of
odd-dimensional spheres divided by an action of the group $U(1)$ of
unitary complex numbers. This presentation allows for a
noncommutative deformation coming from the world of
compact quantum groups via Soibelman-Vaksman spheres.
This approach has been widely explored,
and recently entered the very heart of noncommutative geometry via
the study of Dirac operators on the
thus obtained quantum projective
spaces~\cite{AndreaDabrowski:Projective}.

Recall that
the C*-algebra $C(\mathbb{C}P^{N}_q)$ of functions on a quantum
projective space,
as defined by  Soibelman and
Vaksman~\cite{VaksmanSoibelman:OddSpheres}, is the invariant subalgebra
for an action of $U(1)$ on the C*-algebra of the
odd-dimensional quantum sphere $C(S_q^{2N+1})$
(cf.~\cite{Meyer:QuantumGrassmanians}). By analyzing the space of
 characters, we want to show that this
C*-algebra is not isomorphic to the C*-algebra
$C(\B{P}^{N}(\mathcal{T}))$
of the Toeplitz
quantum projective space proposed in this paper,
unless $N=0$.
To this end, we observe first that
one can easily see from Definition~\ref{qcpn}
that the space of characters
on $C(\B{P}^{N}(\mathcal{T}))$
contains the $N$-torus.  On the other hand,
since $C(\mathbb{C}P_q^{N})$ is
a graph C*-algebra \cite{HongSzymanski:QuantumSpheres},
its space of characters is at most a circle. Hence these
C*-algebras can coincide only for $N=0,1$. For $N=0$, they both
degenerate to $\mathbb{C}$, and for $N=1$, they are known to be
the standard
Podle\'s and mirror quantum spheres, respectively. The latter
are non-isomorphic, so that the claim follows.

Better still,
 one can easily show that the C*-algebras
of the quantum-group projective spaces admit only one character.
Indeed,
 these C*-algebras
 are obtained by iterated extensions
by the ideal of compact operators, i.e., for any
$N$,
there is the  short exact sequence  of C*-algebras
\cite[eq.~4.11]{HongSzymanski:QuantumSpheres}:
\begin{equation}\label{exsentsz}
0\longrightarrow \Cpct\longrightarrow C(\mathbb{C}P_q^{N})
\longrightarrow C(\mathbb{C}P_q^{N-1})\longrightarrow 0.
\end{equation}
On the other hand, any character on a C*-algebra
containing the ideal $\Cpct$
of compact operators must evaluate to $0$ on $\Cpct$, as
 otherwise it would define a proper ideal
in $\Cpct$, which is impossible. Therefore, not only any character on
 $C(\mathbb{C}P^{N-1}_q)$ naturally extends to a character on
$C(\mathbb{C}P^{N}_q)$, but also any character on
$C(\mathbb{C}P^{N}_q)$ naturally descends to a character on
$C(\mathbb{C}P^{N-1}_q)$. Hence the space of characters on
$C(\mathbb{C}P_q^{N})$ coincides with
the space of characters on $C(\mathbb{C}P_q^{N-1})$.
Remembering that
$C(\mathbb{C}P_q^{0})=\mathbb{C}$, we conclude the claim.

\subsection{Noncommutative projective schemes}

Projective spaces \`a la
 Artin-Zhang~\cite{ArtinZhang:ProjectiveSchemes} and
 Rosenberg~\cite{Rosenberg:NCSchemes} are based on Gabriel's
 Reconstruction Theorem~\cite[Ch.~VI]{Gabriel:AbelianCategories}
 (cf.~\cite{Rosenberg:Reconstruction}) and Serre's
 Theorem~\cite[Prop.~7.8]{Serre:ProjectiveSchemes} (cf.~\cite[Vol.~II,
 3.3.5]{Grothendieck:EGA}). The former theorem describes how to
reconstruct   a scheme from its category of
 quasi-coherent sheaves. The latter establishes how to
obtain the category of quasi-coherent
 sheaves over the projective scheme corresponding to a conical affine
 scheme. First, one
 constructs a graded algebra $A$  of
 polynomials on this conical affine scheme and then,
 according to Serre's recipe, one divides the
 category of graded $A$-modules by the subcategory of graded modules
 that are torsion.   Such graded algebras corresponding to
projective manifolds have
 finite global dimension, admit a dualizing module, and their
 Hilbert series have polynomial growth. All this means that they are,
so called, {\em Artin--Schelter
   regular algebras}, or AS-regular algebras in short
 ~\cite{ArtinSchelter:AlgebrasOfGlobalDimension3}
 (cf.~\cite{ArtinTateVandenbergh:ASRegularRings}).
This property makes sense for  algebras which are not necessarily
commutative, so that we think about noncommutative algebras of this
sort as of generalized noncommutative projective manifolds.
 One important
 subclass of such well-behaving algebras are Sklyanin
 algebras~\cite{Sklyanin:YangBaxter}. Among other nice
 properties, they are quadratic Koszul, have finite Gelfand-Kirillov
 dimension~\cite{SheltonTingey:ASRegularAlgebrasAreKoszul}, and are
 Cohen-Macaulay~\cite{Levasseur:ASRegular}.  Another class of
 AS-regular algebras worth mentioning is the class of hyperbolic
rings~\cite{Rosenberg:NCSchemes}, which are also known as
   generalized Weyl algebras~\cite{Bavula:GeneralizedWeylAlgebras}, or
 as  generalized Laurent polynomial
   rings~\cite{CassidyGoetzShelton:LaurentOre}.

\subsection{Quantum deformations of Grassmanian and flag varieties}

In \cite{TaftTowber:QuantumGrassmanians}, Taft and Towber develop a
direct approach to quantizing the Grassmanians, or more generally,
flag varieties.  They define a particular deformation of  algebras
of functions on the classical Grassmanians and flag varieties using an
explicit (in terms of generators and relations) construction of affine
flag schemes defined by
Towber~\cite{Towber:Grassmanians1,Towber:Grassmanians2}.
Their deformation utilizes
$q$-determinants~\cite[Defn.~1.3]{TaftTowber:QuantumGrassmanians}
(cf.\ \cite[p.~227]{Kassel:QuantumGroups} and
~\cite[p.~312]{KlimykSchmudgen:QuantumGroups}) used to
construct a $q$-deformed version of the exterior
product~\cite[Sect.~2]{TaftTowber:QuantumGrassmanians}. This yields a
class of algebras known as quantum exterior
algebras~\cite{Bergh:QuantumExteriorAlgebras}.
These quantum exterior algebras are different
from Weyl algebras or Clifford algebras. They provide counterexamples
 for a number of homological conjectures for finite
dimensional algebras, even though they are  cohomologically well behaved.
See~\cite[Sect.~1]{Bergh:QuantumExteriorAlgebras} for more details.

{\bf Acknowledgements:} This work is part of the EU-project {\sl Geometry and
symmetry of quantum spaces} PIRSES-GA-2008-230836.
 It was also partially supported by the
Polish Government grants N201 1770 33 (PMH, BZ), 1261/7.PRUE/2009/7 (PMH),
189/6.PRUE/2007/7(PMH),
and the Argentinian grant PICT 2006-00836 (AK). Part of
this article was finished during a visit of AK at the Max Planck
Institute in Bonn.  The Institute support and hospitality are gratefully
acknowledged.  Finally, we are very happy to thank the following people for
discussions and advise:
Paul F.\ Baum, Nigel Higson, Masoud Khalkhali, Tomasz
Maszczyk, Ralf Meyer and Jan Rudnik.

\end{document}